\documentclass[12pt]{amsart}
\usepackage{amssymb,latexsym,url}
\usepackage[all]{xy}

\newcommand{\Br}{\operatorname{Br}}

\newcommand{\Gal}{\operatorname{Gal}}

\newcommand{\Div}{\operatorname{Div}}
\newcommand{\divv}{\operatorname{div}}
\newcommand{\disc}{\operatorname{disc}}
\newcommand{\eps}{\varepsilon}
\newcommand{\G}{\mathcal{G}}
\newcommand{\GL}{\operatorname{GL}}
\newcommand{\ev}{\operatorname{ev}}

\newcommand{\Pic}{\operatorname{Pic}}
\newcommand{\PGL}{\operatorname{PGL}}
\newcommand{\End}{\operatorname{End}}
\newcommand{\op}{\operatorname{op}}
\newcommand{\Mat}{\operatorname{Mat}}

\newcommand{\Kbar}{\overline{K}}

\newcommand{\im}{{\operatorname{im}}}
\newcommand{\isom}{\cong}

\newcommand{\PP}{{\mathbb P}}

\newcommand{\ra}{{\longrightarrow}}

\newcommand{\Z}{{\mathbb Z}}

\newfont{\wncyr}{wncyr10 at 12pt}
\newfont{\wncyrten}{wncyr10 at 10pt}

\newcounter{nootje}
\newenvironment{Proof}{\par\noindent{\sc Proof:}}%
                      {\hspace*{\fill}\nobreak$\Box$\par\medskip}
                       {\hspace*{\fill}\nobreak$\Box$\par\medskip}

\newtheorem{Proposition}{Proposition}[section]
\newtheorem{Theorem}[Proposition]{Theorem}
\newtheorem{Lemma}[Proposition]{Lemma}
\newtheorem{Corollary}[Proposition]{Corollary}

\theoremstyle{definition}

\begin{document}

\title[On some algebras associated to genus one curves]%
{On some algebras associated to genus one curves}

\author{Tom~Fisher}
\address{University of Cambridge,
         DPMMS, Centre for Mathematical Sciences,
         Wilberforce Road, Cambridge CB3 0WB, UK}
\email{T.A.Fisher@dpmms.cam.ac.uk}

\date{25th July 2017}  

\begin{abstract}
Haile, Han and Kuo have studied certain non-commutative algebras
associated to a binary quartic or ternary cubic form.
We extend their construction to pairs of quadratic forms
in four variables, and conjecture a further generalisation to
genus one curves of arbitrary degree. These constructions give
an explicit realisation of an isomorphism relating the 
Weil-Ch\^atelet and Brauer groups of an elliptic curve.
\end{abstract}

\maketitle

\renewcommand{\baselinestretch}{1.1}
\renewcommand{\arraystretch}{1.3}
\renewcommand{\theenumi}{\roman{enumi}}

\section{Introduction}

Let $C$ be a smooth curve of genus one, written as either a double
cover of $\PP^1$ (case $n=2$), or as a plane cubic in $\PP^2$ (case $n=3$),
or as an intersection of two quadrics in $\PP^3$ (case $n=4$).
We write $C = C_f$ where $f$ is the binary quartic form, ternary cubic form, 
or pair of quadratic forms defining the curve.
In this paper we investigate a certain non-commutative algebra $A_f$
determined by $f$.

The algebra $A_f$ was defined in the case $n=2$ by Haile and Han
\cite{HH}, and in the case $n=3$ by Kuo \cite{Kuo}. We simplify some
of their proofs, and extend to the case $n=4$. We also conjecture a 
generalisation to genus one curves of arbitrary degree~$n$. 
The following theorem was already established in \cite{HH, Kuo} 
in the cases $n=2,3$.
We work throughout over a field $K$ of characteristic not 2 or 3.

\begin{Theorem}
\label{thm1}
If $n \in \{2,3,4\}$ then $A_f$ is an Azumaya algebra, 
free of rank $n^2$ over its centre. Moreover the centre of $A_f$
is isomorphic to the co-ordinate ring of $E \setminus \{0_E\}$ where
$E$ is the Jacobian elliptic curve of $C_f$.
\end{Theorem} 

Let $E/K$ be an elliptic curve. A standard argument (see
Section~\ref{sec:galcoh}) shows that the Weil-Ch\^atelet group of $E$
is canonically isomorphic to the quotient of Brauer groups
$\Br(E)/\Br(K)$. For our purposes it is more convenient to write this
isomorphism as
\begin{equation}
\label{myisom}
  H^1(K,E) \isom \ker \left( \Br(E) \stackrel{\ev_0}{\ra} \Br(K) \right). 
\end{equation}
where $\ev_0$ is the map that evaluates a Brauer class at 
$0 \in E(K)$.
The algebras we study explicitly realise this isomorphism.

\begin{Theorem}
\label{thm2}
If $n \in \{2,3,4\}$ then the isomorphism~\eqref{myisom} sends the 
class of $C_f$ to the class of $A_f$.
\end{Theorem}

The following two corollaries were proved in~\cite{HH, Kuo}
in the cases $n=2,3$.

\begin{Corollary}
Let $n \in \{2,3,4\}$.
The genus one curve $C_f$ has a $K$-rational point if and only if the
Azumaya algebra $A_f$ splits over $K$.
\end{Corollary}
\begin{proof} This is the statement that the class of $C_f$ in $H^1(K,E)$
is trivial if and only if the class of $A_f$ in $\Br(E)$ is trivial.
\end{proof}

For $0 \not= P \in E(K)$ we write $A_{f,P}$ for the specialisation
of $A_f$ at $P$. This is a central simple algebra over $K$ 
of dimension $n^2$.

\begin{Corollary}
\label{cor2}
Let $n \in \{2,3,4\}$.
The map $E(K) \to \Br(K)$ that sends $P$ to the class of $A_{f,P}$
is a group homomorphism.
\end{Corollary}
\begin{proof}
By Theorem~\ref{thm2} the Tate pairing $E(K) \times H^1(K,E) \to \Br(K)$
is given by $(P,[C_f]) \mapsto [A_{f,P}]$. This corollary is the statement
that the Tate pairing is linear in the first argument.
\end{proof}

The algebras $A_f$ are interesting for several reasons.
They have been used to 
study the relative Brauer groups of curves
(see \cite{CK, Creutz, HHW, Kuo2}) 
and to compute the Cassels-Tate
pairing (see \cite{FN}). We hope they might also be
used to construct explicit Brauer classes on 
surfaces with an elliptic fibration. This could have important
arithmetic applications, extending for example~\cite{Wittenberg}.

In Sections~\ref{def:Af} and~\ref{sec:centre} we define the 
algebras $A_f$ and describe their centres. 
In Section~\ref{sec:cob} we show that these constructions behave
well under changes of co-ordinates. The proofs of 
Theorems~\ref{thm1} and~\ref{thm2} are given in Sections~\ref{sec:pfthm1}
and~\ref{sec:pf2}.

The hyperplane section $H$ on $C_f$ is a $K$-rational divisor of
degree $n$. Let $P \in E(K)$ where $E$ is the Jacobian of $C_f$.  In
Section~\ref{sec:geom} we explain how finding an isomorphism $A_{f,P}
\isom \Mat_n(K)$ enables us to find a $K$-rational divisor $H'$ on
$C_f$ such that $[H-H'] \mapsto P$ under the isomorphism $\Pic^0(C_f)
\isom E$. In the cases $n=2,3$ our construction involves some of the
representations studied in \cite{BH}.

Nearly all our proofs are computational in nature, and for this we
rely on the support in Magma \cite{magma} for finitely presented
algebras.  We have prepared a Magma file checking all our
calculations, and this is available online.  It would of course be
interesting to find more conceptual proofs of Theorems~\ref{thm1}
and~\ref{thm2}.

\section{The algebra $A_f$}
\label{def:Af}

In this section we define the algebras $A_f$ for $n=2,3,4$, and
suggest how the definition might be generalised to genus one curves of
arbitrary degree. The prototype for these constructions is the
Clifford algebra of a quadratic form.  We therefore start by recalling
the latter, which will in any case be needed for our treatment of the
case $n=2$.  We write $[x,y]$ for the commutator $xy - yx$.

\subsection{Clifford algebras}
\label{sec:clifford}
Let $Q \in K[x_1, \ldots,x_n]$ be a quadratic form.  The Clifford
algebra of $Q$ is the associative $K$-algebra $A$ generated by $u_1,
\ldots, u_n$ subject to the relations deriving from the formal
identity in $\alpha_1, \ldots, \alpha_n$,
\[(\alpha_1 u_1 + \ldots + \alpha_n u_n)^2 = Q(\alpha_1, \ldots,
\alpha_n).\] The involution $u_i \mapsto -u_i$ resolves $A$ into
eigenspaces $A = A_+ \oplus A_-$. By diagonalising $Q$, it may be
shown that $A$ and $A_+$ are $K$-algebras of dimensions $2^n$ and
$2^{n-1}$.  Moreover, rescaling $Q$ does not change the isomorphism
class of $A_+$.

In the case $n=3$ we let
\[ \eta = u_1 u_2 u_3 - u_3 u_2 u_1 = u_2 u_3 u_1 - u_1 u_3 u_2 = u_3
u_1 u_2 - u_2 u_1 u_3. \] Then $\eta$ belongs to the centre $Z(A)$,
and $\eta^2 = \disc Q$, where if $Q(x) = x^T M x$ then $\disc Q =
-4\det M$. Moreover, if $\disc Q \not=0$ then $A_+$ is a quaternion
algebra and $A = A_+ \otimes K[\eta]$.  Although not needed below, it
is interesting to remark that the well known map
\[ H^1(K,\PGL_2) \to \Br(K) \] is realised by sending the smooth conic
$\{Q = 0\} \subset \PP^2$ (which as a twist of $\PP^1$ corresponds to
a class in $H^1(K,\PGL_2)$) to the class of $A_+$.

\subsection{Binary quartics}
\label{AfBQ}
Let $f \in K[x,z]$ be a binary quartic, say
\[ f(x,z) = a x^4 + b x^3 z + c x^2 z^2 + d x z^3 + e z^4. \]
Haile and Han \cite{HH} define the algebra $A_f$ to be the
associative $K$-algebra generated by $r,s,t$
subject to the relations deriving from the formal identity 
in $\alpha$ and $\beta$,
\[   (\alpha^2 r + \alpha \beta s + \beta^2 t)^2 = f(\alpha,\beta). \]
Thus $A_f = K\{r,s,t\}/I$ where $I$ is the ideal generated by
the elements
\begin{align*}
& r^2 - a, \\
& rs + sr - b, \\
& rt + tr + s^2 - c,\\
& st + ts - d, \\
& t^2 - e.
\end{align*}

We have $[r,s^2] = [r,rs+sr] = [r,b] = 0$, and likewise
$[s^2,t]=0$. Therefore $\xi = s^2 - c$ belongs to the centre $Z(A_f)$.
By working over the polynomial ring $K[\xi]$, instead of the 
field $K$, we may describe $A_f$ as the Clifford algebra 
of the quadratic form
\[ Q_\xi(x,y,z) = a x^2 + b x y + c y^2 + d y z + e z^2 + \xi(y^2 - xz). \]
This quadratic form naturally arises as follows.
Let $C \subset \PP^3$ be the image of the curve $Y^2 = f(X,Z)$
embedded via $(x_1:x_2:x_3:x_4) = (X^2:XZ:Z^2:Y)$. 
Then $C$ is defined by a pencil of quadrics with 
generic member
$x_4^2 = Q_\xi(x_1,x_2,x_3)$.

\subsection{Ternary cubics}
\label{AfTC}
Let $f \in K[x,y,z]$ be a ternary cubic, say
\begin{align*}
f(x,y,z) & =  a x^3 + b y^3 + c z^3 + a_2 x^2 y + a_3 x^2 z \\
& ~\qquad \quad + \, b_1 x y^2 
+ b_3 y^2 z + c_1 x z^2  + c_2 y z^2 + m x y z. 
\end{align*}
In the special case $c=1$ and $a_3 = b_3 = c_1 = c_2 = 0$, 
Kuo \cite{Kuo} defines the algebra $A_f$ to be the 
associative $K$-algebra generated by $x$ and $y$ subject to
the relations deriving from the formal identity in $\alpha$ and
$\beta$,
\[   f(\alpha,\beta,\alpha x+ \beta y) = 0.  \]
We make the same definition for any ternary cubic 
$f$ with $c \not= 0$. 
Thus $A_f = K\{x,y\}/I$ where $I$ is the ideal generated by
the elements
\begin{align*}
& c x^3 + c_1 x^2 + a_3 x + a,\\
& c (x^2 y + x y x + y x^2) + c_1 (x y + y x) + c_2 x^2 + m x + a_3 y + a_2,\\
& c (x y^2 + y x y + y^2 x) + c_2 (x y + y x) + c_1 y^2 + m y + b_3 x + b_1,\\
& c y^3 + c_2 y^2 + b_3 y + b. 
\end{align*}

\subsection{Quadric intersections}
\label{AfQI}
Let $f=(f_1,f_2)$ be a pair of quadratic forms in four variables, say
$x_1, \ldots, x_4$.  Assuming $C_f = \{f_1 = f_2 = 0 \} \subset \PP^3$
does meet the line $\{ x_3=x_4=0 \}$, we define the algebra $A_f$ to
be the associative $K$-algebra generated by $p,q,r,s$ subject to the
relations deriving from the formal identities in $\alpha$ and $\beta$,
\begin{align}
\label{reln4-1}
   f_i(\alpha p + \beta r,\alpha q + \beta s,\alpha,\beta) &= 0, \quad i=1,2\\
\label{reln4-3}
     [\alpha p + \beta r, \alpha q + \beta s] &= 0.
\end{align}
Explicitly if $f_1 = \sum_{i \le j} a_{ij} x_i x_j$ and 
$f_2 = \sum_{i \le j} b_{ij} x_i x_j$ then
$A_f = K \{ p,q,r,s\}/I$ where $I$ is the ideal generated by the elements
\begin{align*}
&  a_{11} p^2 + a_{12} p q + a_{22} q^2 + a_{13} p + a_{23} q + a_{33}, \\
&  a_{11} (p r + r p) + a_{12} (p s + r q) + a_{22} (q s + s q) + a_{14} p + a_{24} q + a_{13} r + a_{23} s + a_{34}, \\
&  a_{11} r^2 + a_{12} r s + a_{22} s^2 + a_{14} r + a_{24} s + a_{44}, \\
&  b_{11} p^2 + b_{12} p q + b_{22} q^2 + b_{13} p + b_{23} q + b_{33}, \\
&  b_{11} (p r + r p) + b_{12} (p s + r q) + b_{22} (q s + s q) + b_{14} p + b_{24} q + b_{13} r + b_{23} s + b_{34}, \\
&  b_{11} r^2 + b_{12} r s + b_{22} s^2 + b_{14} r + b_{24} s + b_{44}, \\
&  p q - q p, \\
&  p s + r q - q r - s p, \\
&  r s - s r.
\end{align*} 
One motivation for including the commutator relation \eqref{reln4-3}
is that without it, the relations \eqref{reln4-1} would be ambiguous.

\subsection{Genus one curves of higher degree}
Let $C$ be a smooth curve of genus one. If $D$ is a $K$-rational
divisor on $C$ of degree $n \ge 3$ then the complete linear system
$|D|$ defines an embedding $C \to \PP^{n-1}$. We identify $C$ with its
image, which is a curve of degree $n$.  If $n = 3$ then $C$ is a plane
cubic, whereas if $n \ge 4$ then the homogeneous ideal of $C$ is
generated by quadrics.

Let $A$ be the associative $K$-algebra generated by $u_1,u_2, \ldots,
u_{n-2},v_1,v_2, \ldots, v_{n-2}$, subject to the relations deriving
from the formal identities in $\alpha$ and $\beta$,
\begin{align*}
   f(\alpha u_1 + \beta v_1,\alpha u_2 + \beta v_2, \ldots, 
   \alpha u_{n-2} + \beta v_{n-2},\alpha,\beta) &= 0 \quad \text{ for all }
f \in I(C),\\
     [\alpha u_i + \beta v_i, \alpha u_j + \beta v_j] &= 0 
  \quad \text{ for all } 1 \le i,j \le n-2.
\end{align*}
This definition may be thought of as writing down the conditions for
$C$ to contain a line. The fact that $C$ does not contain a line then
tells us that there are no non-zero $K$-algebra homomorphisms $A \to
K$.

We conjecture that the analogues of Theorems~\ref{thm1} and~\ref{thm2}
hold for these algebras.  In support of this conjecture, we have
checked that Theorem~\ref{thm1} holds in some numerical examples with
$n=5$.

\section{The centre of $A_f$}
\label{sec:centre}

In this section we exhibit some elements $\xi$ and $\eta$ in the
centre of $A_f$. In each case 
$\xi$ and $\eta$ generate the centre, and satisfy a relation in the
form of a Weierstrass equation for the Jacobian elliptic curve.

\subsection{Binary quartics} 
\label{ZBQ}
Let $C_f$ be a smooth curve of genus one defined as a double cover of
$\PP^1$ by $y^2 = f(x,z)$, where $f$ is a binary quartic.  It already
follows from the results in Sections~\ref{sec:clifford} and~\ref{AfBQ}
that the centre of $A_f$ is generated by $\xi = s^2 - c$ and $\eta = r
s t - t s r$.  Alternatively, this was proved by Haile and Han
\cite{HH} for quartics with $b=0$, and the general case follows by
making a change of co-ordinates (see Section~\ref{sec:cob}).  The
elements $\xi$ and $\eta$ satisfy $\eta^2 = F(\xi)$ where
\begin{equation}
\label{jac2cubic}
 F(x) =  x^3 + c x^2 - (4 a e - bd) x - 4 ace + b^2 e + ad^2. 
\end{equation}
This is a Weierstrass equation for the Jacobian of $C_f$.

There is a derivation $D: A_f \to A_f$ defined on the generators
$r,s,t$ by $Dr = [s,r]$, $Ds = [t,r]$ and $Dt=0$. To see this is well
defined, we checked that the derivation acts on the ideal of relations
defining $A_f$.  It is easy to see that $D$ must act on the centre of
$A_f$.  We find that $D \xi = 2 \eta$ and $D \eta = 3 \xi^2 + 2 c \xi
- (4 a e - b d)$.

\subsection{Ternary cubics} 
\label{ZTC}
Let $C_f \subset \PP^2$ be a smooth curve of genus one defined by a
ternary cubic $f$. 
With notation as in Section~\ref{AfTC},
the centre of $A_f$ contains
\[ \xi 
    = c^2 (x y)^2   - (c y^2 + c_2 y + b_3) (c x^2 + c_1 x + a_3) 
                               + (c m - c_1 c_2) x y + a_3 b_3.  \]
There is a derivation $D: A_f \to A_f$ defined on the
generators $x,y$ by $Dx = c[xy,x]$ and $Dy = c[y,yx]$.
Let $a'_1, a'_2,a'_3,a'_4, a'_6 \in \Z[a,b,c, \ldots, m]$ be
the coefficients of a Weierstrass equation for the Jacobian of $C_f$,
as specified in \cite[Section 2]{minred}, i.e.
$a'_1 =  m$, $a'_2 =  -(a_2 c_2+a_3 b_3+b_1 c_1)$, 
$a'_3 =  9 a b c - (a b_3 c_2 + b a_3 c_1 + c a_2 b_1) 
   - (a_2 b_3 c_1 + a_3 b_1 c_2), a'_4 = \ldots$.
(These formulae were originally given in \cite{ARVT}.)
Then $\eta = \frac{1}{2}(D \xi - a'_1 \xi - a'_3)$ is also in the
centre of $A_f$, and these elements satisfy
\begin{equation*}
\eta^2 + a'_1 \xi \eta + a'_3 \eta = \xi^3 + a'_2 \xi^2 + a'_4 \xi + a'_6. 
\end{equation*}

In fact $\xi$ and $\eta$ generate the centre of $A_f$. This was proved
by Kuo \cite{Kuo} in the case $c=1$ and $a_3 = b_3 = c_1 = c_2 = 0$. 
The general case follows by making a change of
co-ordinates (see Section~\ref{sec:cob}). 

\subsection{Quadric intersections}
\label{ZQI}
Let $C_f \subset \PP^3$ be a smooth curve of genus one defined by 
a pair of quadratic forms $f=(f_1,f_2)$.
Let $a_1, \ldots, a_{10}$ and $b_1, \ldots, b_{10}$ be the coefficients
of $f_1$ and $f_2$, where we take the monomials in the order
\[ x_1^2,\,\,\, x_1 x_2,\,\,\, x_1 x_3,\,\,\, x_1 x_4,\,\,\, x_2^2,\,\,\, 
   x_2 x_3,\,\,\, x_2 x_4,\,\,\, x_3^2,\,\,\, x_3 x_4,\,\,\, x_4^2. \]
Let $d_{ij} = a_i b_j - a_j b_i$. With notation as in Section~\ref{AfQI}
we put 
\begin{align*}
p_i &= d_{1i} p + d_{2i} q + d_{3i}, &
r_i &= d_{1i} r + d_{2i} s + d_{4i}, \\
q_i &= d_{2i} p + d_{5i} q + d_{6i}, &
s_i &= d_{2i} r + d_{5i} s + d_{7i} 
\end{align*}
and $t = qr - ps = rq - sp$. Then 
\begin{align*}
\xi = (& p_5  s)^2 + (s_1  p)^2 \\
    &+ (d_{56} p_4 + d_{29} p_5 + d_{37} p_5 - d_{27} p_6) s 
    - d_{56} (d_{13} r + d_{23} s - d_{17} q + d_{12} t - d_{19}) s \\
    & + (d_{14} s_6 + d_{29} s_1 - d_{37} s_1 - d_{23} s_4) p 
     - d_{14} (d_{27} p + d_{57} q + d_{35} r - d_{25} t - d_{59}) p
\end{align*}
belongs to the centre of $A_f$. We give a slightly simpler expression
for $\xi$ in Section~\ref{sec:cob4}, but this alternative 
expression is only valid when 
$t$ is invertible.

There is a derivation $D : A_f \to A_f$ defined on the generators
$p,q,r,s$ by $Dp =\tfrac{1}{2}[p,\eps]$, $Dq = \tfrac{1}{2}[q,\eps]$
and $Dr = Ds=0$ where
 \[ \eps
    = d_{12} (p r + r p) 
    + d_{15} (p s + q r + s p + r q) + d_{25} (q s + s q). \]
Then $\eta = \tfrac{1}{2}D\xi$ is also in the centre of $A_f$.
We show in Section~\ref{sec:pfthm1} that $\xi$ and $\eta$ generate the
centre, and that they satisfy
a Weierstrass equation for the Jacobian of~$C_f$.

\section{Changes of co-ordinates}
\label{sec:cob}

In this section we show that making a change of coordinates
does not change the isomorphism class of the algebra $A_f$.
We also describe the effect this has on the central elements $\xi$ and $\eta$, 
and on the derivation $D$.

Let $\G_2(K) = K^\times \times \GL_2(K)$ act on the space of
binary quartics via
\[ (\lambda,M): f(x,z) \mapsto \lambda^2 
     f(m_{11} x + m_{21} z, m_{12} x + m_{22} z). \]
Let $\G_3(K) = K^\times \times \GL_3(K)$ act on the space of
ternary cubics via
\[ (\lambda,M): f(x,y,z) \mapsto \lambda 
     f(m_{11} x + m_{21} y + m_{31}z , \ldots, m_{13} x + m_{23} y + m_{33}z ). \]
Let $\G_4(K) = \GL_2(K) \times \GL_4(K)$ act on the space of
quadric intersections via
\begin{align*}
 (\Lambda,I_4): (f_1,f_2) &\mapsto (\lambda_{11} f_1 + \lambda_{12} f_2,
\lambda_{21} f_1 + \lambda_{22} f_2 ), \\
 (I_2,M): (f_1,f_2) &\mapsto 
     (f_1(\textstyle\sum_{i=1}^4 m_{i1} x_i, \ldots),
     f_2(\textstyle\sum_{i=1}^4 m_{i1} x_i, \ldots)).
\end{align*}
We write $\det(\lambda,M) = \lambda \det M$ in the cases $n=2,3$, and
$\det(\Lambda,M) = \det \Lambda \det M$ in the case $n=4$.  A {\em
  genus one model} is a binary quartic, ternary cubic, or pair of
quadratic forms, according as $n = 2,3$ or $4$.

\begin{Theorem}
\label{thm:cob}
Let $f$ and $f'$ be genus one models of degree $n \in \{2,3,4\}$.
In the case $n=3$ we suppose that $f(0,0,1) \not= 0$ and $f'(0,0,1) \not= 0$.
In the case $n=4$ we suppose that $C_f$ and $C_{f'}$ do not
meet the line $\{ x_3 = x_4 = 0 \}$. 
If $f' = \gamma f$ for some $\gamma \in \G_n(K)$ then 
there is an isomorphism $\psi : A_{f'} \to A_f$ 
with
\begin{align} 
\label{tr:xi}
       \xi &\mapsto (\det \gamma)^2 \xi + \rho \\
\label{tr:eta}
\eta &\mapsto (\det \gamma)^3 \eta + \sigma \xi + \tau
\end{align}
for some $\rho,\sigma,\tau \in K$, with $\sigma=\tau=0$ if $n \in \{2,4\}$. 
Moreover there exists $\kappa \in A_f$ such that
\begin{equation}
\label{tr:D}
 \psi D(x) = (\det \gamma) D \psi (x) + [\kappa,\psi(x)] 
\end{equation}
for all $x \in A_{f'}$. 
\end{Theorem}

\begin{Proof}
  We prove the theorem for $\gamma$ running over a set of generators
  for $\G_n(K)$. The set of generators will be large enough that the
  extra conditions in the cases $n=3,4$ (avoiding a certain point or
  line) do not require special consideration. 

  Writing $\eta$ in terms of the $D \xi$ we see that~\eqref{tr:eta} is
  a formal consequence of \eqref{tr:xi} and \eqref{tr:D}. It therefore
  suffices to check \eqref{tr:xi} and \eqref{tr:D}.  We may
  paraphrase~\eqref{tr:D} as saying that $\psi D \psi^{-1}$ and $(\det
  \gamma) D$ are equal up to inner derivations.  In particular we only
  need to check this statement for $x$ running over a set of
  generators for $A_f$.

We now split into the cases $n=2,3,4$.

\subsection{Binary quartics}
Let $\gamma = (\lambda,M)$. There is an isomorphism $\psi : A_{f'} \to A_f$
given by
\begin{align*}
r &\mapsto \lambda(m_{11}^2 r + m_{11} m_{12} s + m_{12}^2 t), \\
s &\mapsto \lambda(2 m_{11} m_{21} r 
      + (m_{11} m_{22} + m_{12} m_{21}) s + 2 m_{12} m_{22} t), \\
t &\mapsto \lambda(m_{21}^2 r + m_{21} m_{22} s + m_{22}^2 t). 
\end{align*}
We find that~\eqref{tr:xi} and~\eqref{tr:D} are satisfied with
\begin{align*}
\rho = -\lambda^2  \big( 2 & m_{11}^2 m_{21}^2 a 
    + m_{11} m_{21} (m_{11} m_{22} + m_{12} m_{21}) b  
    \\  &   + 2 m_{11} m_{12} m_{21} m_{22} c 
    + m_{12} m_{22} (m_{11} m_{22} + m_{12} m_{21}) d + 2 m_{12}^2 m_{22}^2 e \big) 
\end{align*}
and $\kappa = \lambda (m_{11} m_{21} r + m_{12} m_{21} s + m_{12} m_{22} t)$. 

\subsection{Ternary cubics}
The result is clear for $\gamma = (\lambda,I_3)$.  We take $\gamma =
(1,M)$.  If this change of co-ordinates fixes the point $(0:0:1)$,
equivalently $m_{31} = m_{32} = 0$, then there is an isomorphism $\psi
: A_{f'} \to A_f$ given by
\begin{align*}
  x & \mapsto m_{33}^{-1} ( m_{11} x + m_{12} y - m_{13} ), \\
  y & \mapsto m_{33}^{-1} ( m_{21} x + m_{22} y - m_{23} ).
\end{align*}
We checked~\eqref{tr:xi} by a generic calculation (leading to a
lengthy expression for $\rho$ which we do not record here), and find
that~\eqref{tr:D} is satisfied with
\[ \kappa = c m_{33} \big(m_{23} (m_{11} x + m_{12} y) - m_{13}
(m_{21} x + m_{22} y) \big). \] 

It remains to consider a transformation that moves the point
$(0:0:1)$.  Let $f'(x,y,z) = f(z,x,y)$. By hypothesis $a = f'(0,0,1)
\not= 0$. From the first relation defining 
$A_f$ it follows
that $x$ is invertible, 
i.e. $x^{-1} = -(c x^2 + c_1 x + a_3)/a$.  
There is an isomorphism $\psi : A_{f'} \to A_f$ given by $x
\mapsto -yx^{-1}$ and $y \mapsto x^{-1}$. 
We find that~\eqref{tr:xi} and~\eqref{tr:D} are satisfied 
with $\rho = 0$ and $\kappa = c y x + c_1 y$.

\subsection{Quadric intersections}
\label{sec:cob4}
The result for $\gamma = (\Lambda,I_4)$ follows easily from the fact
our expressions for $\eps$ and $\xi$ are linear and quadratic 
in the $d_{ij}$. 
We take $\gamma = (I_2,M)$. If 
\[ M = \begin{pmatrix} U^{-1} & 0 \\ 0 & I_2 \end{pmatrix}
\quad \text{ or } \quad 
\begin{pmatrix} I_2 & 0 \\ V & I_2 \end{pmatrix} \]
then an isomorphism $\psi : A_{f'} \to A_f$ is given by
\[ \left\{ \begin{aligned}
p &\mapsto u_{11} p + u_{21} q \\
 q &\mapsto u_{12} p + u_{22} q \\
r &\mapsto u_{11} r + u_{21} s \\ 
s &\mapsto u_{12} r + u_{22} s
\end{aligned} \right\}
\quad \text{ or } \quad 
\left\{ \begin{aligned}
p &\mapsto p - v_{11} \\ 
q &\mapsto q - v_{12} \\
r &\mapsto r - v_{21} \\
 s &\mapsto s - v_{22}
\end{aligned} \right\}. \]
We checked~\eqref{tr:xi} by a generic calculation, and find 
that~\eqref{tr:D} is satisfied with $\kappa = 0$ or 
$\kappa = v_{11} (d_{12} r + d_{15} s) + v_{12} (d_{15} r + d_{25} s)$.

It remains to consider a transformation that moves the line $\{x_3 = x_4=0\}$.
Let $f_i'(x_1,x_2,x_3,x_4) = f_i(x_3,x_4,x_1,x_2)$ for $i=1,2$.
By hypothesis $C_f$ does not meet the line $\{x_1 = x_2=0\}$ and so
$t=qr-ps$ is invertible, i.e.
\[  t^{-1} =  -(d_{89} (s_1 r + s_4)
    + d_{8,10} (r_5 q + r_6 + p_5 s + p_7 + d_{29}) 
         + d_{9,10} (q_1 p + q_3))/\Delta\]
where $\Delta = d_{8,10}^2 - d_{89} d_{9,10}$.
There is an isomorphism $\psi:A_{f'} \to A_f$ given by
$p \mapsto -st^{-1}$, $q \mapsto q t^{-1}$, $r \mapsto r t^{-1}$,
$s \mapsto -p t^{-1}$ and $t \mapsto t^{-1}$.
Under our assumption that $t$ is invertible, we have 
$\xi = \xi_1 + c_1$ where 
\begin{align*}
\xi_1 =& (d_{15}^2 - d_{12} d_{25}) t^2
      + (d_{15} d_{37} - d_{12} d_{67} - d_{15} d_{46} - d_{25} d_{34}) t \\
     & + (d_{37} d_{8,10} - d_{36} d_{9,10} + d_{46} d_{8,10} - d_{47} d_{89}) t^{-1}
      + (d_{8,10}^2 - d_{89} d_{9,10}) t^{-2},
\end{align*}
and $c_1 \in K$ is a constant (depending on $f$). 
Working with $\xi_1$ in place of $\xi$ makes it easy
to check~\eqref{tr:xi}. Finally~\eqref{tr:D} is satisfied with
\[ \kappa =
   \lambda \big(p (s_1 r + s_4) + q_{10}\big) 
   +  \mu \big(r (q_1 p + q_3) + s_8\big) 
  + r (d_{12} p + d_{15} q)
    - \tfrac{1}{2}(d_{23} r + d_{26} s) \]
for certain $\lambda, \mu \in K$. In fact we may take
$\lambda = (2 d_{48} d_{8,10}  - d_{38} d_{9,10}  
  - d_{89} d_{49} + d_{89} d_{3,10})/(2 \Delta)$
and $\mu = (2 d_{3,10} d_{8,10} - d_{4,10} d_{89}  
   - d_{9,10} d_{39} + d_{9,10}  d_{48})/(2 \Delta)$.
\end{Proof}

\section{Proof of Theorem~\ref{thm1}}
\label{sec:pfthm1}

In this section we prove the following refined version of 
Theorem~\ref{thm1}. The first two parts of the
theorem show that $A_f$ is an Azumaya algebra.
\begin{Theorem}
\label{thm1:alt}
Let $C_f$ be a smooth curve of genus one, defined by a genus one
model $f$ of degree $n \in \{2,3,4\}$. Then
\begin{enumerate}
\item The algebra $A = A_f$ is free of rank $n^2$ over its centre 
$Z$ (say).
\item The map $A \otimes_Z A^{\op} \to \End_Z(A); \, 
a \otimes b \mapsto (x \mapsto a xb)$ is an isomorphism.
\item The centre $Z$ is generated by the elements $\xi$ and $\eta$
specified in Section~\ref{sec:centre}, subject only to these satisfying 
a Weierstrass equation. 
\item The Weierstrass equation in (iii) defines the
Jacobian of $C_f$.
\end{enumerate}
\end{Theorem}

For the proof of the first three parts of Theorem~\ref{thm1:alt}
we are free to extend our field $K$. However working over
an algebraically closed field, it is well known that smooth curves 
of genus one $C_f$ and $C_{f'}$ are isomorphic as curves (i.e. have 
the same $j$-invariant) if and only if the genus one models $f$ 
and $f'$ are in the same orbit for the group action defined at
the start of Section~\ref{sec:cob}. We now split into the cases
$n=2,3,4$ and verify the theorem by direct computation for a family
of curves covering the $j$-line. The general case then follows
by Theorem~\ref{thm:cob}.

The generic calculations in Sections~\ref{ZBQ} and~\ref{ZTC}
already prove Theorem~\ref{thm1:alt}(iv) in the cases $n=2,3$.
The case $n=4$ will be treated in Section~\ref{sec:pf4}.

\subsection{Binary quartics}
Let $K[x_0,y_0] = K[x,y]/(F)$ where 
\[ F(x,y) = y^2 - (x^3 + a_2 x^2 + a_4 x + a_6). \]
We consider the binary quartic
$f(x,z) = a_6 x^4 + a_4 x^3 z + a_2 x^2 z^2 + x z^3$. 
Specialising the formulae in Section~\ref{ZBQ} we see that 
$\xi,\eta \in A_f$ satisfy $F(\xi,\eta)=0$.
\begin{Lemma}
There is an isomorphism of $K$-algebras
$\theta : A_f \to \Mat_2(K[x_0,y_0])$ given by
\[ r \mapsto \begin{pmatrix}
               -y_0 & x_0^2 + a_2 x_0 + a_4 \\
                -x_0 &                y_0 
\end{pmatrix}, \quad
 s \mapsto \begin{pmatrix}
0 & x_0 + a_2 \\
1 & 0 
\end{pmatrix}, \quad
 t \mapsto \begin{pmatrix}
0 & 1 \\
0 & 0 
\end{pmatrix}.
\]
Moreover we have $\theta(\xi) = x_0 I_2$ and $\theta(\eta) = y_0 I_2$.
\end{Lemma}
\begin{Proof}
We write $E_{ij}$ for the $2$ by $2$ matrix with a $1$ in the $(i,j)$
position and zeros elsewhere. Then $\Mat_2(K[x_0,y_0])$ is generated
as a $K[x_0,y_0]$-algebra by $E_{12}$ and $E_{21}$ subject to the relations
$E_{12}^2 = E_{21}^2 = 0$ and $E_{12} E_{21} + E_{21} E_{12}= 1$.
We define a $K$-algebra homomorphism $\phi : \Mat_2(K[x_0,y_0]) \to A_f$
via \[x_0 \mapsto \xi, \,\, y_0 \mapsto \eta, \,\, E_{12} \mapsto t, \,\, 
E_{21} \mapsto s - s^2t.\] We checked by direct calculation 
that $\theta$ and $\phi$ are well defined
(i.e. they send all relations to zero), and that they are inverse to 
each other.
\end{Proof}

\subsection{Ternary cubics}
Let $K[x_0,y_0] = K[x,y]/(F)$ where 
\[ F(x,y) = y^2 + a_1 x y + a_3 y - (x^3 + a_2 x^2 + a_4 x + a_6).\]
We consider the ternary cubic $f(x,y,z) = x^3 F(z/x,y/x)$. 
Specialising the formulae in Section~\ref{ZTC} we see that 
$\xi,\eta \in A_f$ satisfy $F(\xi,\eta)=0$.
\begin{Lemma}
There is an isomorphism of $K$-algebras $\theta : A_f 
\to \Mat_3(K[x_0,y_0])$ given by
\[ x \mapsto \begin{pmatrix}
         -x_0 - a_2 &         -1 &     0 \\
    x_0^2 + a_2 x_0 + a_4 &     0 &    y_0 \\
      y_0 + a_1 x_0 + a_3 &     0 &    x_0 
\end{pmatrix},
\qquad y \mapsto 
\begin{pmatrix}
      0 &  0 &  0 \\
    -a_1 & 0 & -1 \\
     1 &   0 &  0 
\end{pmatrix}. \]
Moreover we have $\theta(\xi) = x_0 I_3$ and $\theta(\eta) = y_0 I_3$.
\end{Lemma}
\begin{Proof}
We write $E_{ij}$ for the $3$ by $3$ matrix with a $1$ in the $(i,j)$
position and zeros elsewhere. Then $\Mat_3(K[x_0,y_0])$ is generated
as a $K[x_0,y_0]$-algebra by $E_{12}$, $E_{23}$ and $E_{31}$ subject 
to the relations
\[ E_{12}^2 = E_{23}^2 =  E_{31}^2 = E_{12} E_{31} = E_{23} E_{12} = E_{31} E_{23} = 0, \]
and 
\[
 E_{12} E_{23} E_{31} + E_{23} E_{31} E_{12} + E_{31} E_{12} E_{23} = 1. \]
We define a $K$-algebra homomorphism $\phi : \Mat_3(K[x_0,y_0]) \to A_f$
via $x_0 \mapsto \xi$, $y_0 \mapsto \eta$ and 
\[  E_{12}  \mapsto -x y^2 (x + \xi + a_2), \quad
    E_{23}  \mapsto  -y^2 (x y - a_1), \quad
    E_{31}  \mapsto  (y x - a_1) y^2. \]
We checked by direct calculation 
that $\theta$ and $\phi$ are well defined, and that they are inverse to 
each other.
\end{Proof}

\subsection{Quadric intersections}
\label{sec:pf4}

Let $f' \in K[x,z]$ be a binary quartic, say
\[ f'(x,z) = a x^4 + b x^3 z + c x^2 z^2 + d x z^3 + e z^4. \]
The morphism $C_{f'} \to \PP^3$ given by
$(x_1 : x_2 : x_3 : x_4) = (xz : y : x^2 : z^2)$ has image
$C_f$ where $f = (f_1,f_2)$ and
\begin{equation}
\label{QIfromBQ}
\begin{aligned}
f_1(x_1,x_2,x_3,x_4) & = x_1^2- x_3 x_4, \\
f_2(x_1,x_2,x_3,x_4) & = 
    x_2^2 - (a x_3^2 + b x_1 x_3 + c x_3 x_4 + d x_1 x_4 + e x_4^2).
\end{aligned}
\end{equation}
We write $r',s',t'$ and $\xi',\eta'$ for 
the generators and central elements of $A_{f'}$. 
\begin{Lemma} 
\label{lem:42}
There is an isomorphism of $K$-algebras 
$\theta : A_f \to \Mat_2(A_{f'})$ given by
\[ p \mapsto \begin{pmatrix} 0 & 1 \\ 0 & 0 \end{pmatrix}, \quad 
q \mapsto \begin{pmatrix} r' & s' \\ 0 & r' \end{pmatrix}, \quad
r \mapsto \begin{pmatrix} 0 & 0 \\ 1 & 0 \end{pmatrix}, \quad
s \mapsto \begin{pmatrix} t' & 0 \\ s' & t' \end{pmatrix}.
\]
Moreover we have $\theta(\xi) = (\xi' + c) I_2$ and $\theta(\eta) = -\eta' I_2$.
\end{Lemma}
\begin{proof}
Again the proof is by direct calculation, the 
$K$-algebra homomorphism 
inverse to $\theta$ being given by
$E_{12} \mapsto p$, $E_{21} \mapsto r$ and
\[ 
r' I_2  \mapsto pqr + rqp, 
  \quad 
s' I_2 \mapsto prqr+rqrp, 
  \quad 
t' I_2 \mapsto psr + rsp. \qedhere \]
\end{proof}

To complete the proof of Theorem~\ref{thm1:alt}, and hence of
Theorem~\ref{thm1}, it remains to show that in the case $n=4$
the Weierstrass equation satisfied by $\xi$ and $\eta$ is in fact
an equation for the Jacobian of $C_f$.

Let $A$ and $B$ be the $4$ by $4$ matrices of partial derivatives of
$f_1$ and $f_2$. We define $a,b,c,d,e$ by writing $\tfrac{1}{4}
\det(Ax + B) = a x^4 + b x^3 + c x^2 + d x + e$.  As shown in
\cite{AKM3P}, the Jacobian of $C_f$ has Weierstrass equation $y^2 =
F(x)$ where $F$ is the monic cubic polynomial defined
in~\eqref{jac2cubic}.

We claim that $\eta^2 = F(\xi + c_0)$ for some constant $c_0 \in K$
(depending on $f$).  In verifying this claim we are free to extend our
field.  We are also free to make changes of coordinates. Indeed if $f'
= \gamma f$ for some $\gamma = (\Lambda,M) \in \G_4(K)$ then by
Theorem~\ref{thm:cob} there is an isomorphism $\psi : A_{f'} \to A_f$
with $\xi \mapsto (\det \gamma)^2 \xi + \rho$ and $\eta \mapsto (\det
\gamma)^3 \eta$. On the other hand the monic cubic polynomials $F$ and
$F'$ (associated to $f$ and $f'$) are related by $F'(x-\frac{1}{3}c')
= (\det \gamma)^6 F((\det \gamma)^{-2} x - \frac{1}{3}c)$.
Finally we checked that for $f$ as specified in~\eqref{QIfromBQ},
the claim is satisfied with $c_0 = 0$.
 
\section{Proof of Theorem~\ref{thm2}}
\label{sec:pf2}

In this section we recall the definition of the isomorphism~\eqref{myisom},
and then prove that the construction of $A_f$ from $C_f$ is an explicit
realisation of this map.

\subsection{Galois cohomology}
\label{sec:galcoh}
Let $E/K$ be an elliptic curve. Writing $\Kbar$ for a separable closure
of $K$, the short exact sequences of Galois modules
\begin{equation}
\label{ses1}
 0 \to \Kbar^\times 
  \to \Kbar(E)^\times \to  \Kbar(E)^\times/\Kbar^\times \to 0, 
\end{equation}
and \[ 0 
  \to \Kbar(E)^\times/\Kbar^\times \to \Div E \to \Pic E \to 0, \]
give rise to long exact sequences
\begin{equation}
\label{les1}
 H^2(K,  \Kbar^\times) \to  H^2(K,\Kbar(E)^\times) \to  H^2(K,\Kbar(E)^\times
/\Kbar^\times ), 
\end{equation}
and
\begin{equation}
\label{les2}
 H^1(K, \Div E) \to  H^1(K,\Pic E) \to  H^2(K,\Kbar(E)^\times
/\Kbar^\times ) \to  H^2(K, \Div E). 
\end{equation}
Since $H^1(K,\Z) = 0$ it follows by Shapiro's lemma that $H^1(K, \Div E)=0$.
We may identify $H^1(K,\Pic E) = H^1(K,\Pic^0 E) = H^1(K,E)$
and $H^2(K,  \Kbar^\times) = \Br(K)$.
As shown in \cite[Appendix]{L} we may identify
\[\Br(E) = \ker \left( H^2(K,\Kbar(E)^\times ) \to H^2(K,\Div E ) \right). \]

We fix a local parameter $t$ at $0 \in E(K)$. The left hand map
in~\eqref{ses1} is split by the map sending a Laurent power series in
$t$ to its leading coefficient.  It follows that the right hand map
in~\eqref{les1} is surjective, and hence $H^1(K,E) \isom
\Br(E)/\Br(K)$. Since the natural map $\Br(K) \to \Br(E)$ is split by
evaluation at $0 \in E(K)$ this also gives the
isomorphism~\eqref{myisom}.

\subsection{Cyclic algebras}
Let $L/K$ be a Galois extension with $\Gal(L/K)$ cyclic of order $n$,
generated by $\sigma$. For $b \in K^\times$ the cyclic algebra
$(L/K,b)$ is the $K$-algebra with basis $1,v, \ldots, v^{n-1}$ as an
$L$-vector space, and multiplication determined by $v^n = b$ and $v
\lambda = \sigma(\lambda) v$ for all $\lambda \in L$. This is a
central simple algebra over $K$ of dimension $n^2$. It is split by $L$
and so determines a class in $\Br(L/K)$.

We compute cohomology of $C_n = \langle \sigma | \sigma^n= 1\rangle$
relative to the resolution 
\[ \ldots \ra \Z[C_{n}] \stackrel{\Delta}{\ra}
   \Z[C_{n}] \stackrel{N}{\ra}
   \Z[C_{n}] \stackrel{\Delta}{\ra}
   \Z[C_{n}] \, \ra \, 0, \] 
where $\Delta = \sigma-1$ and $N = 1+ \sigma + \ldots + \sigma^{n-1}$.
Thus for $A$ a $\Gal(L/K)$-module, 
\[  H^i(\Gal(L/K),A) = \left\{
\begin{array}{ll} 
\ker(N|A)/ \im (\Delta|A) & \text{ if $i \ge 1$ odd,} \\
\ker(\Delta|A)/ \im (N|A) & \text{ if $i \ge 2$ even.}
\end{array} \right. \]
In particular $K^\times/N_{L/K}(L^\times) \isom
H^2(\Gal(L/K),L^\times) = \Br(L/K)$. This isomorphism is 
realised by sending $b \in K^\times$ to the class of $(L/K,b)$.

Let $E/K$ be an elliptic curve, and fix a local parameter $t$ at $0
\in E(K)$.  If $g \in K(E)^\times$ then we write $(L/K,g)$ for the
cyclic algebra $(L(E)/K(E),g)$.  We may describe the
isomorphism~\eqref{myisom} in terms of cyclic algebras as follows.

\begin{Lemma} 
\label{lem:cyc}
Let $C/K$ be a smooth curve of genus one curve with Jacobian $E$, and
suppose $Q \in C(L)$.  Let $P$ be the image of $[\sigma Q - Q]$ under
$\Pic^0(C) \isom E$.  Then the isomorphism~\eqref{myisom} sends the
class of $C$ to the class of $(L/K,g)$ where $g \in K(E)^\times$ has
divisor $(P) + (\sigma P) + \ldots + (\sigma^{n-1}P) - n(0)$, and is
scaled to have leading coefficient $1$ when expanded as a Laurent
power series in $t$.
\end{Lemma}
\begin{Proof} 
  We identify $E \isom \Pic^0(E)$ via $T \mapsto (T)-(0)$. Then the
  class of $C$ in $H^1(K,\Pic E)$ is represented by $(P)-(0)$, and its
  image under the connecting map in~\eqref{les2} is represented by $g
  \in K(E)^\times$ where $\divv g = N_{L/K} ((P)-(0))$.  Finally to
  lift to an element of $\ker ( \ev_0 : \Br(E) \to \Br(K) )$ we scale
  $g$ as indicated. 
\end{Proof}

\subsection{Binary quartics}
\label{pf2:BQ}
We prove Theorem~\ref{thm2} in the case $n=2$.
By a change of coordinates we may assume\footnote{It is incorrectly claimed
in \cite[Section 5]{HH} that we may further assume $d=0$.} that $a\not=0$ 
and $b=0$, i.e.
\[ f(x,z) = a x^4 + c x^2 z^2 + d x z^3 + e z^4. \]

Let $E$ be the Jacobian of $C_f$, with Weierstrass equation as
specified in Section~\ref{ZBQ}.  We know by Theorem~\ref{thm1} that
the centre $Z$ of $A_f$ is a Dedekind domain with field of fractions
$K(E)$.  Therefore the natural map $\Br(Z) \to \Br(K(E))$ is
injective, and so it suffices for us to consider the class of the
quaternion algebra $A_f \otimes_Z K(E)$ in $\Br(K(E))$.  This algebra
is generated by $r$ and $s$ subject to the rules $r^2=a$, $rs + sr =
0$ and $s^2 = \xi + c$. It is therefore the cyclic algebra $(L/K,g)$
where $L = K(\sqrt{a})$ and $g \in K(E)^\times$ is the rational
function 
$g(\xi,\eta) = \xi + c$.

By inspection of the Weierstrass equation for $E$ in
Section~\ref{ZBQ}, we see that $\divv g = (P) + (\sigma P) - 2(0)$
where $P = (-c,d\sqrt{a}) \in E(L)$.  Let $C_f$ have equation $y^2 =
f(x_1,x_2)$, and let $Q \in C(L)$ be the point $(x_1:x_2:y) =
(1:0:\sqrt{a})$.  Let $\pi : C_f \to E$ be the covering map, i.e. the
map $T \mapsto [2(T) - H]$ where $H$ is the fibre of the double cover
$C_f \to \PP^1$.  Using the formulae in~\cite{AKM3P} we find that
$\pi(Q) = -P$.  Therefore $[\sigma Q - Q] = [H - 2(Q)] = P$.  It
follows by Lemma~\ref{lem:cyc} that the isomorphism~\eqref{myisom}
sends the class of $C$ to the class of the cyclic algebra $(L/K,g)$.
This completes the proof of Theorem~\ref{thm2} in the case $n=2$.

\subsection{Ternary cubics}
We prove Theorem~\ref{thm2} in the case $n=3$.  Since $2$ and $3$ are
coprime, we are free to replace our field $K$ by a quadratic
extension. We may therefore suppose that $\zeta_3 \in K$ and that
$f(x,0,z) = a x^3 - z^3$ with $a \not=0$.  Further substitutions of
the form $x \leftarrow x + \lambda y$ and $z \leftarrow z + \lambda'
y$ reduce us to the case
\[ f(x,y,z) = a x^3 + b y^3 - z^3 + b_1 x y^2 + b_3 y^2 z + m x y z
. \]

The algebra $A_f \otimes_Z K(E)$ is generated by $x$ and $v = y x -
\zeta_3 x y - \tfrac{1}{3}(1-\zeta_3)m$ subject to the rules $x^3 =
a$, $x v = \zeta_3 v x$ and $v^3 = g(\xi,\eta)$
where 
 \[ g(\xi,\eta) =  \eta - \zeta_3^2 m \xi - 3 (1-\zeta_3) a b 
            + \tfrac{1}{9}(\zeta_3-\zeta_3^2) m^3. \]
It is therefore the cyclic algebra $(L/K,g)$ where $L = K(\sqrt[3]{a})$.

Let $E$ be given by the Weierstrass equation specified in
Section~\ref{ZTC}.  We find that $\divv g = (R) + (\sigma R) +
(\sigma^2 R) - 3(0)$ for a certain point $R \in E(L)$ with
$x$-coordinate $-(1/3) m^2 + b_1 \sqrt[3]{a} - b_3 (\sqrt[3]{a})^2$.
Let $Q = (1:0:\sqrt[3]{a}) \in C_f(L)$.  Let $\pi : C_f \to E$ be the
covering map, i.e. the map $T \mapsto [3(T) - H]$ where $H$ is the
hyperplane section.  Using the formulae in~\cite{AKM3P} we find that
$\pi(Q) = \sigma R - \sigma^2 R$.  We compute
\[ 3 [ \sigma Q - Q] = \pi(\sigma Q) - \pi( Q) 
= \sigma(\sigma R - \sigma^2 R ) - (\sigma R - \sigma^2 R) = 3 \sigma^2 R. \]
Since generically $E$ has no $3$-torsion, it follows that
$[\sigma Q - Q] = \sigma^2 R$. Taking $P = \sigma^2 R$
in Lemma~\ref{lem:cyc} completes the proof.

\subsection{Dihedral algebras}
Let $L/K$ be a Galois extension with $\Gal(L/K) \isom D_{2n}$ where
$D_{2n} = \langle \sigma,\tau | \sigma^n = \tau^2 = (\sigma \tau)^2 =
1 \rangle$ is the dihedral group of order $2n$.  Let $K_1$, $F$ and
$\widetilde{F}$ be the fixed fields of $\sigma$, $\tau$ and $\sigma
\tau$. For $(b,\eps,\widetilde{\eps}) \in K_1^\times \times F^\times
\times \widetilde{F}^\times$ satisfying $N_{K_1/K}(b) N_{F/K}(\eps) =
N_{\widetilde{F}/K}(\widetilde{\eps})$ we define the dihedral algebra
$(L/K,b,\eps,\widetilde{\eps})$ to be the $K$-algebra with basis
$1,v,\ldots, v^{n-1}, w,vw, \ldots, v^{n-1}w$ as an $L$-vector space,
and multiplication determined by $v^n = b$, $w^2 = \eps$, $(vw)^2 =
\widetilde{\eps}$, $v \lambda = \sigma(\lambda) v$ 
and $w \lambda = \tau(\lambda) w$ for all $\lambda \in L$.  As
we explain below, this is a special case of a crossed product
algebra. In particular it is a central simple algebra over $K$ of
dimension $(2n)^2$. It is split by $L$ and so determines a class in
$\Br(L/K)$.

Let $N = 1 + \sigma + \ldots +\sigma^{n-1} \in \Z[D_{2n}]$.  We
compute cohomology of $D_{2n}$ relative to the resolution
\begin{equation}
\label{dihres}
 \ldots \ra \Z[D_{2n}]^4 \stackrel{\Delta_3}{\ra}
   \Z[D_{2n}]^3 \stackrel{\Delta_2}{\ra}
   \Z[D_{2n}]^2 \stackrel{\Delta_1}{\ra}
   \Z[D_{2n}] 
     \, \ra \, 0 
\end{equation}
where 
\[ \Delta_3 = \begin{pmatrix} \sigma - 1 & 0 & 0 \\
0 & \tau - 1 & 0 \\ 0 & 0& \sigma \tau - 1 \\ 
\tau + 1 & N & -N
\end{pmatrix}, \,\,
\Delta_2 = \begin{pmatrix} N & 0 \\ 
0 & \tau + 1 \\
\sigma \tau + 1 & \sigma + \tau 
\end{pmatrix}, \,\, \Delta_1 = \begin{pmatrix} \sigma - 1 \\ \tau -
  1 \end{pmatrix}, \] and our convention is that $\Delta_m$ acts by
right multiplication on row vectors.  This resolution is a special
case of that defined in~\cite{Wall}, except that we have applied some
row and column operations to simplify $\Delta_2$ and $\Delta_3$.
Using this resolution to compute $\Br(L/K)=H^2(\Gal(L/K),L^\times)$ we
find
\begin{equation}
\label{Br-isom}
\frac{ \{ (b,\eps,\widetilde{\eps}) \in K_1^\times
\times F^\times \times  \widetilde{F}^\times \mid N_{K_1/K}(b) N_{F/K}(\eps)
 = N_{\widetilde{F}/K}(\widetilde{\eps}) \} }{ \{ (N_{L/K_1}(\lambda_1),
N_{L/F}(\lambda_2), N_{L/\widetilde{F}}(\lambda_1 \lambda_2)) \mid 
\lambda_1, \lambda_2 \in L^\times \} }  \isom \Br(L/K). 
\end{equation}
This isomorphism is realised by sending $(b,\eps,\widetilde{\eps})$ to
the class of the dihedral algebra $(L/K,b,\eps,\widetilde{\eps})$.
Our claim that dihedral algebras
are crossed product algebras is justified by comparing this 
description of $\Br(L/K)$ with that obtained from the standard resolution.

In more detail, there is a commutative diagram of free $\Z[D_{2n}]$-modules
\[ \xymatrix{ \ar[r] & \bigoplus_{(g,h) \in D_{2n}^2} \Z[D_{2n}]
\ar[d]^{\phi_2} \ar[r]^-{d_2} & \bigoplus_{g \in D_{2n}} \Z[D_{2n}] 
\ar[r]^-{d_1} \ar[d]^{\phi_1} & \Z[D_{2n}] \ar@{=}[d] \ar[r] & 0 \\
\ar[r] & \Z[D_{2n}]^3 \ar[r]^{\Delta_2} & \Z[D_{2n}]^2 \ar[r]^{\Delta_1}
& \Z[D_{2n}] \ar[r] & 0 } \]
where the first row is the standard resolution, i.e. $d_1(e_g) = g-1$
and $d_2(e_{g,h}) = g(e_h) - e_{gh} + e_g$, and the second row is 
the resolution~\eqref{dihres}. We choose $\phi_1$ such that
\begin{align*}
\phi_1(e_1) &= (0,0), & \phi_1(e_{\sigma^i}) &= (1 + \sigma + \ldots
 + \sigma^{i-1},0) \quad \text{ for } 0<i<n, \\
\phi_1(e_\tau) &= (0,1), & \phi_1(e_{\sigma^i \tau}) &= (1 + \sigma + \ldots
 + \sigma^{i-1}, \sigma^i) \quad \text{ for } 0<i<n. 
\end{align*}
We further choose $\phi_2$ such that for $0 \le i,j <n$ we have
\[ \phi_2(e_{\sigma^i,\sigma^j}) = \phi_2(e_{\sigma^i,\sigma^j \tau})
= \left\{ \begin{array}{ll} (0,0,0) & \text{ if } i+j<n, \\
(1,0,0) & \text{ if } i+j \ge n, \end{array} \right. \]
and $\phi_2(e_{\tau,\tau}) = (0,1,0)$, $\phi_2(e_{\sigma \tau, \sigma \tau})
= (0,0,1)$. The $2$-cocycle $\xi \in Z^2(D_{2n}, L^\times)$ corresponding
to $(b,\eps,\widetilde{\eps})$ is now the unique $2$-cocycle satisfying
\[ \xi_{\sigma^i,\sigma^j} = \xi_{\sigma^i,\sigma^j \tau} = 
\left\{ \begin{array}{ll} 1 & \text{ if } i+j <n, \\ b & \text{ if }
i+j \ge n, \end{array} \right. \]
and $\xi_{\tau,\tau} = \eps$, $\xi_{\sigma\tau,\sigma\tau} = \widetilde{\eps}$.
The cross product algebra associated to $\xi$ is the $K$-algebra with
basis $\{ v_g : g \in D_{2n} \}$ as an $L$-vector space, and multiplication
determined by $v_g v_h = \xi_{g,h} v_{gh}$ and $v_g \lambda = g(\lambda) v_g$
for all $\lambda \in L$. Identifying $v_{\sigma^i} = v^i$ and
$v_{\sigma^i \tau} = v^i w$, we recognise this as the dihedral algebra
$(L/K,b,\eps,\widetilde{\eps})$.

Let $E/K$ be an elliptic curve, and fix a local parameter $t$ at $0
\in E(K)$.  We may describe the isomorphism~\eqref{myisom} in terms of
dihedral algebras as follows.
\begin{Lemma} 
\label{lem:dih}
Let $C/K$ be a smooth curve of genus one with Jacobian $E$, and
suppose $Q \in C(F)$.  Let $P$ be the image of $[\sigma Q - Q]$ under
$\Pic^0(C) \isom E$.  Then the isomorphism~\eqref{myisom} sends the
class of $C$ to the class of $(L/K,g,1,h)$ where $g \in K_1(E)^\times$
and $h \in \widetilde{F}(E)^\times$ have divisors $(P) + (\sigma P) +
\ldots + (\sigma^{n-1}P) - n(0)$ and $(P) + (-P) - 2(0)$, and are
scaled to have leading coefficient $1$ when expanded as Laurent power
series in $t$.
\end{Lemma}
\begin{Proof} 
  We have $[\sigma Q - Q] = P$ and $[\tau Q - Q] = 0$.  We identify $E
  \isom \Pic^0(E)$ via $T \mapsto (T)-(0)$. Then the class of $C$ in
  $H^1(K,\Pic E)$ is represented by the pair $((P)-(0),0)$.  Reading
  down the first column of $\Delta_2$, the image of this class under
  the connecting map in~\eqref{les2} is represented by a triple
  $(g,1,h)$ where $\divv g = N_{L/K_1} ((P)-(0))$ and $\divv h =
  (\sigma \tau + 1) ((P)-(0)) = (P) + (-P) - 2(0)$.  Finally to lift
  to an element of $\ker ( \ev_0 : \Br(E) \to \Br(K) )$ we scale $g$
  and $h$ as indicated. 
\end{Proof}

\subsection{Quadric intersections}
We prove Theorem~\ref{thm2} in the case $n=4$.  We are free to make
field extensions of odd degree.  We may therefore suppose that $C_f$
meets the plane $\{x_4 = 0\}$ in four points in general position, and
that one of the three singular fibres in the pencil of quadrics
vanishing at these points is defined over $K$.  In other words, we may
assume that $f_1(x_1,x_2,x_3,0) = q_1(x_1,x_3)$ where $q_1$ is a
binary quadratic form. Then $f_2$ must have a term $x_2^2$, and so by
completing the square $f_2(x_1,x_2,x_3,0) = x_2^2 +
q_2(x_1,x_3)$. Adding a suitable multiple of $f_1$ to $f_2$ we may
suppose that $q_2$ factors over $K$, and so without loss of generality
$q_2(x_1,x_3) = - x_1 x_3$. Making linear substitutions of the form
$x_i \leftarrow x_i + \lambda x_4$ for $i=1,2,3$ brings us to the case
\begin{align*}
f_1(x_1,x_2,x_3,x_4) &= a x_1^2 + b x_1 x_3 + c x_3^2 + (d_1 x_1 + d_2 x_2
+ d_3 x_3 + d_4 x_4) x_4, \\
f_2(x_1,x_2,x_3,x_4) &= x_2^2 - x_1 x_3 - e x_4^2. 
\end{align*} 

Let $L/K$ be the splitting field of $G(X) = a X^4 + b X^2 + c$.  Then
$\Gal(L/K)$ is a subgroup of $D_8$. We suppose it is equal to $D_8$,
the other cases being similar.  We have $L = K(\theta,\sqrt{\delta})$
where $\theta$ is a root of $G$ and $\delta = ac(b^2 - 4ac)$. The
generators $\sigma$ and $\tau$ of $D_8$ act as
\begin{align*}
\sigma : \theta &\mapsto \frac{1}{\sqrt{\delta}} ( ab \theta^3 
  + (b^2 - 2 a c) \theta), & 
\sigma : \sqrt{\delta} &\mapsto \sqrt{\delta}, \\ 
\tau : \theta &\mapsto \theta, & \tau :  \sqrt{\delta} &\mapsto -\sqrt{\delta}.
\end{align*}
The fixed fields of $\sigma$, $\tau$ and $\sigma \tau$ are
$K_1 = K(\sqrt{\delta})$, $F = K(\theta)$ and $\widetilde{F} = K(\phi)$
where $\phi = a(\theta + \sigma(\theta))$.

Let $A = A_f \otimes_Z K(E)$. The second generator $q$ of $A_f$
satisfies $a q^4 + b q^2 + c = 0$. We may therefore embed $F \subset
A$ via $\theta \mapsto q$, and hence $L \subset A_1 = A \otimes_K
K_1$.  We find that $A_1$ is generated as a $K_1(E)$-algebra by $q$
and
\[
v = a (r \sigma(q) - q r) + \frac{a}{\sqrt{\delta}} (a q^3 \sigma(q) 
       - c) (d_1 q + d_2 + d_3 q^{-1})
 \]
 subject to the rules $a q^4 + b q^2 + c = 0$, $v q = \sigma(q) v$ and
 $v^4 = g(\xi,\eta)$, for some $g \in K_1(E)$.  It is therefore the
 cyclic algebra $(L/K_1,g)$. Writing $\xi$ for the element that was
 denoted $\xi + c_0$ in Section~\ref{sec:pf4}, we have
\[ g(\xi,\eta) = 
    \xi^2 - \frac{4 a c d_2}{\sqrt{\delta}} \eta
          + \frac{2 (b m + 2 a c d_2^2)}{b^2 - 4 a c} \xi
          + 8 a c e d_2^2 + \frac{m^2 + d_2^2 n}{b^2 - 4 a c}.
\]
where 
$m = c d_1^2 - b d_1 d_3 + a d_3^2 + (b^2 - 4 a c) d_4$ 
and
$n = b c d_1^2 + a c (d_2^2 - 4 d_1 d_3) + a b d_3^2$.

Let $Q = (\theta^2 : \theta : 1 : 0) \in C_f(F)$, and let $P = [\sigma
Q - Q]$ under the usual identification $\Pic^0(C_f) \isom E$.  We
compute the point $P$ as follows. We put
\[ \begin{pmatrix} Tz_1 \\ z_2 \\ z_1 \\ T z_2 \end{pmatrix}
= \begin{pmatrix} 
  a & \phi & (\phi^2 + ab)/2a & 0 \\
  a & -\phi & (\phi^2 + ab)/2a & 0 \\
  -a d_1 & -a d_2 & - a d_3 & -a d_4 - e \phi^2 \\
  0 & 0 & 0 & 1 
\end{pmatrix} 
\begin{pmatrix} x_1 \\ x_2 \\ x_3 \\ x_4 \end{pmatrix}. \]
Inverting this $4$ by $4$ matrix $M$ gives
\[ (\det M) f_2 (x_1,x_2,x_3,x_4) = \alpha(z_1,z_2) T^2 +
\beta(z_1,z_2)T + \gamma(z_1,z_2) \] for some binary quadratic forms
$\alpha, \beta, \gamma$.  Relacing $f_2$ by $f_1$ gives a scalar
multiple of the same equation.  Therefore $C_f$ has equation $y^2 =
\beta(z_1,z_2)^2 - 4 \alpha(z_1,z_2) \gamma(z_1,z_2)$. The points $Q$
and $\sigma(Q)$ are given by $(z_1:z_2:y) = (1:0:\pm a(\theta -
\sigma(\theta)))$.  Exactly as in Section~\ref{pf2:BQ}, we compute $P$
using the formulae for the covering map. Relative to the Weierstrass
equation for $E$ specified in Section~\ref{sec:pf4}, this point has
$x$-coordinate
\begin{equation}
\label{eqn:xP}
x(P) = \frac{2 a \phi^2 (d_2 d_3 \phi + m) - d_2 
    (d_1 \phi + a d_2) (b \phi^2 + a (b^2 - 4 a c))}{2 a^2 (b^2-4 a c)}
\in \widetilde{F}.
\end{equation}

We find that $P$ and its Galois conjugates are zeros of $g$.
Therefore $\divv g = (P) + (\sigma P) + (\sigma^2 P) + (\sigma^3 P) -
4(0)$.  It follows by Lemma~\ref{lem:cyc} that the class of $A_f$, and
the image of the class of $C_f$ under~\eqref{myisom}, agree after
restricting to $\Br (E \otimes_K K_1)$.  It remains to show that the
same conclusion holds without the quadratic extension.

For $a \in A_1 = A \otimes_K K_1$ let $\overline{a} = (1 \otimes
\tau)a$.  We find that $v \overline{v} =
\xi - x(P)$ where $x(P)$ is given by~\eqref{eqn:xP}.  Now let $A_2 =
A_1 \oplus A_1 w$ with multiplication determined by $w^2 = 1$ and $w a
= \overline{a} w$ for all $a \in A_1$. This is the dihedral algebra
$(L/K,g,1,\xi-x(P))$.  The subalgebra generated by $K_1$ and $w$ is a
trivial cyclic algebra. Therefore $A_2 \isom A \otimes_K \Mat_2(K)$.
In particular $A$ and $A_2$ have the same class in
$\Br(K(E))$. Lemma~\ref{lem:dih} now completes the proof.

\section{Geometric interpretation}
\label{sec:geom}

Let $C$ be a smooth curve of genus one with Jacobian elliptic curve
$E$.  Let $H$ and $H'$ be $K$-rational divisors on $C$ of degree $n
\ge 2$.  We assume that $H$ and $H'$ are {\em not} linearly
equivalent, and so their difference corresponds to a non-zero point $P
\in E(K)$.  The complete linear systems $|H|$ and $|H'|$ define an
embedding $C \to \PP^{n-1} \times \PP^{n-1}$. Assuming $n \in
\{2,3,4\}$, the composite of this map with the first and second
projections is described by genus one models $f$ and $f'$.

In this section we investigate the following problem. 
\begin{quotation} 
Given $f$ and
$P$, how can we compute $f'$? 
\end{quotation}
The answers we give might be viewed as explicitly realising the
connection between the Tate pairing and the obstruction map, as
studied in \cite{descsumI, CON, Zarhin}. Our answers also serve to
motivate the definition of $A_f$, and indeed (however much it might
seem an obvious guess in hindsight) this is how we actually found the
correct definition of $A_f$ in the case $n=4$.

We give no proofs in this section. However all our claims may be verified
by generic calculations.

\subsection{Binary quartics}
The image of $C \to \PP^1 \times \PP^1$ is defined by a $(2,2)$-form,
say
\[ F(x,z; x',z') = f_1(x,z) x'^2 + 2 f_2(x,z) x' z' + f_3(x,z) z'^2. \]
Then $f = f_2^2 - f_1 f_3$, and $f'$ is obtained in the same way, 
after switching the two sets of variables. Thus, given a binary
quartic $f$, we seek to find binary quadratic forms $f_1,f_2,f_3$ such
that
\[ \begin{pmatrix} f_2 & -f_1 \\ f_3 & - f_2 \end{pmatrix}^2
= \begin{pmatrix} f & 0 \\ 0 & f \end{pmatrix}. \] Equivalently, we
look for matrices $M_1, M_2, M_3 \in \Mat_2(K)$ satisfying
\[  (\alpha^2 M_1 + \alpha \beta M_2 + \beta^2 M_3)^2 = f(\alpha,\beta)I_2.\]

This reduces the problem of finding $f'$ from $f$ to that of finding a
$K$-algebra homomorphism $A_f \to \Mat_2(K)$.  By Theorem~\ref{thm1}
any such homomorphism must factor via $A_{f,P}$ for some $0 \not= P
\in E(K)$.  This point $P$ turns out to be the same as the point $P$
considered at the start of Section~\ref{sec:geom}.  In conclusion, if
$A_{f,P} \isom \Mat_2(K)$ and we can find this isomorphism explicitly,
then we can write down a $(2,2)$-form, and hence a binary quartic
$f'$, such that $C_f$ and $C_{f'}$ are isomorphic as genus one curves,
but their hyperplane sections differ by $P$.

\subsection{Ternary cubics}
The image of $C \to \PP^2 \times \PP^2$ is defined by three
$(1,1)$-forms.  The coefficients may be arranged as a $3 \times 3
\times 3$ cube.  As explained in \cite{BH}, slicing this cube in three
different ways gives rise to three ternary cubics. Two of these are
$f$ and $f'$.  Thus, given a ternary cubic $f$, we seek to find
matrices $M_1, M_2, M_3 \in \Mat_3(K)$ satisfying
\[ f(\alpha,\beta,\gamma) = \det(\alpha M_1 + \beta M_2 + \gamma
M_3). \] If $f(0,0,1) \not= 0$ then we may assume (after rescaling $f$
and multiplying each $M_i$ on the left by the same invertible matrix)
that $M_3 = -I_3$. Then $\alpha M_1 + \beta M_2$ has characteristic
polynomial $\gamma \mapsto f(\alpha,\beta,\gamma)$, and so by the
Cayley-Hamilton theorem
\[ f(\alpha,\beta,\alpha M_1 + \beta M_2) = 0. \]

This reduces the problem of finding $f'$ from $f$ to that of finding a
$K$-algebra homomorphism $A_f \to \Mat_3(K)$.  By Theorem~\ref{thm1}
any such homomorphism must factor via $A_{f,P}$ for some $0 \not= P
\in E(K)$.  This point $P$ turns out to be the same as the point $P$
considered at the start of Section~\ref{sec:geom}.  In conclusion, if
$A_{f,P} \isom \Mat_3(K)$ and we can find this isomorphism explicitly,
then we can write down a $3 \times 3 \times 3$ cube, and hence a
ternary cubic $f'$, such that $C_f$ and $C_{f'}$ are isomorphic as
genus one curves, but their hyperplane sections differ by $P$.

\subsection{Quadric intersections}
The image of $C \to \PP^3 \times \PP^3$ is defined by an
$8$-dimensional vector space $V$ of $(1,1)$-forms in variables
$x_1,\ldots,x_4$ and $y_1, \ldots, y_4$. Let $W$ be the vector space
of $4$ by $4$ alternating matrices $B = (b_{ij})$ of linear forms in
$y_1, \ldots,y_4$ such that
\[ \sum_{i=1}^4 x_i b_{ij} (y_1, \ldots, y_4) \in V \quad \text{ for
  all } j= 1, \ldots, 4. \] We find that $W$ is $4$-dimensional. We
choose a basis, and let $M$ be a generic linear combination of the
basis elements, say with coefficients $z_1, \ldots, z_4$. Then $M =
(m_{ij})$ is a $4$ by $4$ alternating matrix of $(1,1)$-forms in $y_1,
\ldots,y_4$ and $z_1, \ldots,z_4$. The Pfaffian of this matrix is a
$(2,2)$-form, which turns out to be
\[ f_1^+(y_1,\ldots,y_4)f_2^-(z_1,\ldots,z_4) - f_2^+(y_1,\ldots,y_4)
f_1^-(z_1,\ldots,z_4), \] where $f^{\pm} = (f_1^\pm,f_2^\pm)$
describes the image of $C \to \PP^3$ via $|H^\pm|$, and $[H-H^\pm] =
\pm P$.  To tie in with our earlier notation, $H^+ = H'$ and $f^+ =
f'$.

We write $m_{ij} = (y_1, \ldots ,y_4) M_{ij} (z_1, \ldots, z_4)^T$
where $M_{ij} \in \Mat_4(K)$. Assuming $C_f$ does not meet the line
$\{x_3 = x_4 = 0\}$ we have $\det(M_{12}) \not= 0$, and so we may
choose our basis for $W$ such that $M_{12} = I_4$. The matrices
$M_{ij}$ then satisfy $f_i(\alpha M_{23} + \beta M_{24}, - (\alpha
M_{13} + \beta M_{14}), \alpha,\beta) = 0$ for $i=1,2$, where the
first two arguments commute, and $M_{34} = M_{13} M_{24} - M_{23}
M_{14} = M_{24} M_{13} - M_{14} M_{23}$.

This reduces the problem of finding $f'$ from $f$ to that of finding a
$K$-algebra homomorphism $A_f \to \Mat_4(K)$.  By Theorem~\ref{thm1}
any such homomorphism must factor via $A_{f,P}$ for some $0 \not= P
\in E(K)$. Again this point $P$ turns out to correspond to the
difference of hyperplane sections for $C_f$ and $C_{f'}$.


\begin{thebibliography}{MM}

\frenchspacing
\renewcommand{\baselinestretch}{1}

\bibitem{AKM3P}
S.Y. An, S.Y. Kim, D.C. Marshall, S.H. Marshall,
W.G. McCallum and A.R. Perlis,
Jacobians of genus one curves,
{\em J. Number Theory} {\bf {90}} (2001), no. 2, 304--315. 

\bibitem{ARVT}
M. Artin, F. Rodriguez-Villegas and J. Tate, 
On the Jacobians of plane cubics,
{\em Adv. Math.} {\bf {198}} (2005), no. 1, 366--382.

\bibitem{BH} 
M. Bhargava and W. Ho, 
Coregular spaces and genus one curves, {\em Camb. J. Math.} {\bf{4}}
(2016), no. 1, 1--119.

\bibitem{magma}
W. Bosma, J. Cannon and C. Playoust, 
The Magma algebra system I: The user language, {\em J. Symb. Comb.} {\bf{24}}, 
235--265 (1997). 
\url{http://magma.maths.usyd.edu.au/magma/}

\bibitem{CK}
M. Ciperiani and D. Krashen, 
Relative Brauer groups of genus 1 curves,
{\em Israel J. Math.} {\bf 192} (2012), no. 2, 921--949.

\bibitem{descsumI}
J.E. Cremona, T.A. Fisher, C. O'Neil, D. Simon and M. Stoll, 
Explicit $n$-descent on elliptic curves, 
I Algebra, {\em J. reine angew. Math.}  {\bf 615} (2008) 121--155.

\bibitem{minred}
J.E. Cremona, T.A. Fisher and M. Stoll, 
Minimisation and reduction of 2-, 3- and 4-coverings of elliptic curves,
{\em Algebra \& Number Theory} {\bf 4} (2010), no. 6, 763--820.

\bibitem{Creutz}
B. Creutz, 
Relative Brauer groups of torsors of period two,
{\em J. Algebra} {\bf 459} (2016), 109--132.

\bibitem{FN}
T.A. Fisher and R.D. Newton, 
Computing the Cassels-Tate pairing on the 3-Selmer group of an
elliptic curve, {\em Int. J. Number Theory} {\bf 10} (2014), no. 7,
1881--1907.

\bibitem{HH}
D. Haile and I. Han, 
On an algebra determined by a quartic curve of genus one,
{\em J. Algebra} {\bf 313} (2007), no. 2, 811--823.

\bibitem{HHW}
D.E. Haile, I. Han and A.R. Wadsworth, 
Curves $C$ that are cyclic twists of $Y^2=X^3+c$ and the relative 
Brauer groups $\Br(k(C)/k)$,
{\em Trans. Amer. Math. Soc.} {\bf 364} (2012), no. 9, 4875--4908.

\bibitem{Kuo}
J.-M. Kuo, 
On an algebra associated to a ternary cubic curve. 
{\em J. Algebra} {\bf 330} (2011), 86--102.

\bibitem{Kuo2}
J.-M. Kuo, 
On cyclic twists of elliptic curves of period two or 
three and the determination of their relative Brauer groups,
{\em J. Pure Appl. Algebra} {\bf 220} (2016), no. 3, 1206--1228.

\bibitem{L}
S. Lichtenbaum, 
Duality theorems for curves over $p$-adic fields.
{\em Invent. Math.} {\bf{7}} (1969) 120--136.

\bibitem{CON}
C. O'Neil, 
The period-index obstruction for elliptic curves.
{\em J. Number Theory} {\bf 95} (2002), no. 2, 329--339.

\bibitem{Wall}
C.T.C. Wall,
Resolutions for extensions of groups,
{\em Proc. Cambridge Philos. Soc.} {\bf{57}} (1961) 251--255.

\bibitem{Wittenberg}
O. Wittenberg, 
Transcendental Brauer-Manin obstruction on a pencil of elliptic curves,
Arithmetic of higher-dimensional algebraic varieties (Palo Alto, CA, 2002), 
259--267, Progr. Math., 226, Birkh\"auser Boston, Boston, MA, 2004.
 
\bibitem{Zarhin}
Ju. G. Zarhin, 
Noncommutative cohomology and Mumford groups. 
{\em Mat. Zametki} {\bf 15} (1974), 415--419; English translation:  
{\em Math. Notes} {\bf 15} (1974), 241--244. 

\end{thebibliography}
\end{document}